\documentclass[11pt]{article}
\usepackage{latexsym,color,amsmath,amsthm,amssymb,amscd,amsfonts}

\newtheorem{lemma}{Lemma}

\newtheorem{remark}[lemma]{Remark}

\newtheorem*{remark*}{Remark}

\def\upddots{\mathinner{\mkern 1mu\raise 1pt \hbox{.}\mkern 2mu
\mkern 2mu \raise 4pt\hbox{.}\mkern 1mu \raise 7pt\vbox {\kern 7
pt\hbox{.}}} }

\newcommand{\msl}{\widetilde{{G}}}
\newcommand{\mb}{\widetilde{{B}}}
\newcommand{\mm}{\widetilde{{M}}}
\newcommand{\mh}{\widetilde{{H}}}
\newcommand{\mgl}{\widetilde{{GL_2(\F)}}}
\newcommand{\mslt}{{\widetilde{{G}}}}

\newcommand{\glt}{{GL_2(\F)}}

\newcommand{\slt}{{SL_2}(\F)}

\newcommand{\F}{F}

\newcommand{\Of}{\mathbb O_{\F}}
\newcommand{\Pf}{\mathbb P_{\F}}

\newcommand{\N}{\mathbb N}

\newcommand{\half}{\frac 1 2}

\newcommand{\ab} {|\!|}

\newcommand{\Q}{\mathbb Q}
\newcommand{\C}{\mathbb C}

\def\>{\rangle}
\def\<{\langle}

\newtheorem{lem}{Lemma}
\newtheorem{thm}{Theorem}
\newtheorem{cor}{Corollary}

\newtheorem{prop}{Proposition}
\def\dotunion{
\def\dotunionD{\bigcup\kern-9pt\cdot\kern5pt}
\def\dotunionT{\bigcup\kern-7.5pt\cdot\kern3.5pt}
\mathop{\mathchoice{\dotunionD}{\dotunionT}{}{}}} \newtheorem*{theorem*}{Theorem}
\newtheorem*{lemma*}{Lemma}

\begin{document}
%\title{The dimension of the space of Whittaker functionals  on constituents of   reducible genuine unitary principal series representation of the $n$ fold cover of $\slt$ }

\title{Whittaker spaces for reducible unitary principal series representations of $\widetilde{SL_2(F)}$}
\author{Dani Szpruch}
 \maketitle
\begin{abstract}
Let $F$ be a $p$-adic field containing the full group of $n^{th}$ roots of 1 and let $\msl$ be the $n$-fold cover of $SL_2(F)$ constructed by Kubota. In this paper we compute the dimension of the space of Whittaker functionals of the two irreducible summands  inside a reducible unitary genuine principal series representation of $\msl$. We also show how these dimensions change when the  Whittaker character is modified. As an application we determine the action of the twisted Kazhdan-Patterson  $n$-fold cover of  $GL_2(F)$ on the two summands. We emphasize that our main results addresses both ramified and unramified representations and do not rely on the assumption that the cover is tame.
\end{abstract}
\section{Introduction}
Let $G$ be a quasi-split reductive group defined over a $p$-field $F$. Fix a Borel subgroup ${ B}={ H} \ltimes { N}$ and a non-degenerate character $\psi$ of $N$. Let $M$ be a Levi subgroup of $G$ containing $H$ and let $\psi'$ be the restriction of $\psi$ to $N\cap M$.

Let $\sigma$ be a smooth admissible  irreducible  representation of $M$ and let $Wh_{\psi'}(\sigma)$ be the space of $\psi'$-Whittaker functionals on $\sigma$. $\sigma$ is called $\psi'$-generic if  $Wh_{\psi'}(\sigma)$ is non-trivial. The well known uniqueness of Whittaker model stats that if $\sigma$ is $\psi'$-generic then $\dim Wh_{\psi'}(\sigma) = 1$. Let $\pi=I(\sigma)$ be the representation of $G$ parabolically induced from $\sigma$. Rodier Heredity  implies that if $\sigma$ is $\psi'$-generic then $\dim Wh_\psi(\pi)=1$. In particular, if $\pi$ is reducible then exactly one of its irreducible constituents is $\psi$-generic.

This situation is changed dramatically when we move to covering groups. For these groups, uniqueness of Whittaker model rarely holds, see \cite{GSS}. In recent years the problem of determining  the space of Whittaker functionals of various representations of covering groups has been studied extensively. See for example \cite{Gao16}, \cite{Gao2021},  \cite{GaoGurKar2} \cite{GSS3}, \cite{Zou}.

If the uniqueness of Whittaker model fails for a certain covering group, it fails already for its genuine principal series representations. The goal of this paper is to determine the dimensions of the space of Whittaker functionals for the irreducible constitutions of reducible unitary genuine principal series representations of $\msl=\widetilde{SL_2(F)}$, the $n$-fold cover of $SL_2(F)$ defined in \cite{Kub}. We emphasize that our results apply to both ramified and unramified representations and to any Whittaker character. We also  emphasize  that we do not rely on the assumption that the cover is tame. Namely, we do not assume that $gcd(n,p)=1$.

Such reducible representations of $\msl$ exists only if $n$ is odd, see \cite{Sz19}, Proposition 5.5 and Section 5.7. Similar to the linear case, the set of
reducible unitary genuine principal series representations of $\msl$ is parameterized by the set of non-trivial quadratic characters of $F^*$, see Section \ref{planchsec} below.

Let $\pi$ be a reducible genuine unitary principal series representation of $\msl$. It is a sum of two non-isomorphic irreducible representations. Since
$\dim Wh_\psi(\pi)$ is odd we may write $\pi=\pi^{+}_\psi \oplus \pi^{-}_\psi$ where the irreducible summands $\pi^{\pm}_\psi$ are defined by the relation
$$\dim Wh_\psi (\pi^{+}_\psi)>\dim Wh_\psi (\pi^{-}_\psi).$$
Our main result given in Section \ref{mainsec}, Theorem \ref{mainthm}, states that

$$\dim Wh_\psi (\pi^{+}_\psi)=\ab n \ab^{-\half} \frac{n+1}{2}, \, \, \dim Wh_\psi (\pi^{-}_\psi)=\ab n \ab^{-\half} \frac{n-1}{2}.$$

Here $\ab \cdot \ab$ is the usual $p$-adic absolute value. Moreover, if one changes the Whittaker character to $\psi_c$ where $c \in F^*$ (see the body of the paper for details) then

$$\dim Wh_{\psi_c} (\pi^{+}_\psi)=\ab n \ab^{-\half} \frac{n+\theta(c)}{2}, \, \, \dim Wh_{\psi_c} (\pi^{-}_\psi)=\ab n \ab^{-\half} \frac{n-\theta(c)}{2}.$$
Here $\theta$ is the non-trivial quadratic character of $F^*$ associated with $\pi$.

We now give an outline of the proof of Theorem 1: first we normalize the standard intertwining operator by a suitable $\gamma$-factor so that the normalized
operator is a $\msl$ isomorphism of order 2. We then use the normalized operator to construct the projections from $Wh_\psi(\pi)$ to $Wh_\psi(\pi_\psi^{\pm})$. The last step in the proof is a computation of the rank of these projections. The computation ultimately boils down to a computation of the trace of a certain scattering matrix. A by-product of our proof is that $\pi_\psi^{\pm}$ are identified as the $\pm 1$ eigenspaces of the normalized intertwining operator.

Our results in the unramified case are among Gao's results in \cite{Gao2021}. The normalization factor in that work of Gao is the one coming from the Gindikin-Karpelevich formula and his computation of the trace uses the explicit description of the scattering matrices obtained by McNamara in \cite{Mc2}. Both Gindikin-Karpelevich formula and McNamara's work are applicable for unramified representations only. In fact McNamara assumes that the Whittaker character is normalized.

Our first novelty is the  usage the of the explicit computation of the Plancherel measure in \cite{Sz19}, Theorem 5.7 to obtain the normalization factor, see Section \ref{planchsec} below. In the unramified case it agrees with Gao's. Our second novelty is the usage of the coarse description of the local coefficients matrices in \cite{Sz19}, Theorem 4.12 to compute the trace. Using the formalism from \cite{Sz19}, this last computation turns out to be very simple, in fact it is nothing but Fourier inversion for finite abelian groups, see Section \ref{scatsec}.

This paper is organized as follows. First, in Section \ref{presec} we introduce some known notions and results regarding  the $p$-adic field $F$ and the covering group $\msl$ along with some material on its principal series representations. Our main result is presented in Section \ref{mainresec}. We have included two remarks in the end of Section \ref{mainsec}: in Remark \ref{mainrem} we discuss the analogy to the linear case and  the relation to the theory of $R$-groups. Remark \ref{Sharem} addresses the relation between the normalization factor and the Shahidi-type invariant defined in \cite{GSS}.

In Section \ref{glsec} we use our  main result to determine the action of the Kazhdan-Patterson twisted $n$-fold cover of $GL_2(F)$ on $\pi_\psi^{\pm}$. Finally, in Section \ref{unramsec} we use this action along with our main result to study the unramified case.
\section{Preliminaries} \label{presec}
\subsection{The $p$-adic field} \label{padicsec}
Let $F$ be a finite extension of $\Q_p$. Denote by $q$ the cardinality of its residue field. Denote by $\Of$ its ring of integers. Fix $\varpi$, a generator of $\Pf$, the maximal ideal of $\Of$. We normalize the absolute value on $F$ such that $\ab \varpi \ab =q^{-1}$.

Let $\psi$ be a non-trivial character of $F$. We define $e(\psi)$, the conductor of $\psi$, to be the smallest integer $k$ such that $\psi$ is trivial on $\Pf^k$. For $c \in F^*$ let $\psi_c$ be the character of $F$ given by $x \mapsto \psi(cx)$. The map $\psi\mapsto \psi_c$ is an isomorphism from $F$ to its Pontryagin dual $\widehat{F}$. We denote by $d_\psi$ the $\psi$-self dual Haar measure on $F$. %Note that for $c \in F^*$ we have $d_{\psi_c}=\ab c \ab^{\half} d_\psi$ and that $q^{e(\psi_c)}=\ab c \ab q^{e(\psi)}$.

Let $\chi$ be a character of $F^*$. If $\chi$ is ramified we define its conductor, $e(\chi)$, to be the smallest  $k\in \N$ such that $\chi$ is trivial on $1+\Pf^k$. If $\chi$ is unramified we set $e(\chi)=0$. For $s\in \C$ we set $\chi_s=\ab \cdot \ab^s\chi.$ In this paper $\theta$ will denote a non-trivial quadratic character of $F^*$.

Tate $\gamma$-factor associated with $\chi$, $\psi$ and $s\in \C$ is defined by $$\gamma(s,\chi,\psi)=\epsilon(s,\chi,\psi)\frac{L(1-s,\chi^{-1})}{L(s,\chi)}$$
where $\epsilon(s,\chi,\psi)$ is a monomial function in $q^{-s}$ satisfying
\begin{eqnarray} \label{Tate gamma} {\epsilon}(1-s,\chi^{-1},\psi) &=& \chi(-1)\epsilon(s,\chi,\psi)^{-1}, \\
\label{changepsi}\epsilon(s,\chi,\psi_c) &=& \chi(c)\ab c \ab^{s-\half}\epsilon(s,\chi,\psi), \\
\label{epsilon old twist} \epsilon(s+t,\chi,\psi) &=& q^{e(\psi)-e(\chi)t}\epsilon(s,\chi,\psi),\end{eqnarray}
We note that if $e(\psi)=0$ and $\chi$ is unramified then $\epsilon(s,\chi,\psi)=1$. For all these assertions see Section 1 in \cite{Schmidt} for example.  Combining \eqref{Tate gamma} and \eqref{epsilon old twist} one obtains
\begin{equation} \label{epsilon twist and inv}
\epsilon(1-s,\chi^{-1},\psi)\epsilon(1+s,\chi,\psi)= \chi(-1)q^{e(\psi)-e(\chi)}.
\end{equation}

Fix an odd integer $n \in \N$. We shall assume in this paper that $F^*$ contains the full group of $n^{th}$ roots of 1. Denote this cyclic group by $\mu_n$. We identify $\mu_n$ with the group of $n^{th}$ roots of 1 in $\C^*$ and suppress this identification. Let $$( \cdot, \cdot):F^* \times F^* \rightarrow \mu_{n}$$ be the $n^{th}$ power Hilbert symbol. Recall that it is an anti-symmetric bilinear form and that its kernel in each argument is ${F^*}^{n}$. Hence, it gives rise to a non-degenerate bilinear form on $F^* / {F^*}^{n} \times F^* / {F^*}^{n}$. In particular, it identifies $F^* / {F^*}^{n}$ with $\widehat{F^* / {F^*}^{n}}$. The group $\widehat{F^* / {F^*}^{n}}$ may also  be identified with the group of  characters of $F^*$ whose order divide $n$. For $x \in F^*$ let $\eta_x$ be the character of $F^*$ defined by
$$y \mapsto \eta_x(y)=(x,y).$$
%Note that the map $x \mapsto \eta_x$ factors through $F^*/{F^*}^n$.
When convenient we  shall also think of $\eta_x$ as an element in  $\widehat{F^*/{F^*}^n}$.

A subgroup $\overline{J}$ and of $F^*/{F^*}^n$ is called  a maximal isotropic subgroup if   $$\overline{J} =\bigcap_{x \in {\overline{J}}} \operatorname{ker}(\eta_x).$$ By  \cite{Sz19}, Lemma 2.3, there exist two maximal isotropic subgroups $\overline{J}$ and $\overline{K}$ of $F^*/{F^*}^n$ such that
\begin{enumerate}
%\item  and $\overline{K} =\bigcap_{x \in {\overline{K}}} \operatorname{ker}(\eta_x)$.
\item $\overline{J} \times \overline{K}=F^*/{F^*}^n.$
\item The map $k \mapsto {\eta_k \! \mid}_{\overline{J}}$ is an isomorphism from $\overline{K}$ to the Pontryagin dual of $\overline{J}$.
\end{enumerate}
\begin{remark} \label{stanrem} In the case where $gcd(n,p)=1$, the standard choice in the literature is $$(\overline{J},\overline{K})=(\Of^* {{ F}^*}^n / {{ F}^*}^n, <\varpi> {{ F}^*}^n / {{ F}^*}^n).$$
We shall not use this.
\end{remark}

Fix $k \in \overline{K}$. The following partial $\gamma$-factor was  defined in \cite{Sz19}, Page 127,
$$\gamma_{_J}(s,\chi,\psi,k)={(\# \overline{J})}^{-1} \sum_{j \in \overline{J}} \gamma(s,\chi\eta_j,\psi)\eta_k(j).$$
See \cite{Sz19}, Theorem 2.12 for the motivation for this definition. By Fourier inversion for finite abelian groups we have

\begin{equation} \label{invert} \sum_{k \in \overline{K}}\gamma_{_J}(s,\chi,\psi,k)=\gamma(s,\chi,\psi).\end{equation}

\subsection{The covering of $SL_2(F)$}
Most of the material in this section and in Sections \ref{torussec}, \ref{prinsec} and \ref{whidefsec} is well known. Full details and references can be found in \cite{Sz19}.

Let $G={{ SL}_2({ F})}$ be the group of two by two matrices with entries in ${ F}$ whose determinants equal 1. Let ${ N} \simeq { F}$ be the group of upper triangular unipotent matrices in $G$ and let ${ H} \simeq { F}^*$ be the group of diagonal elements inside  $G$. Denote ${ B}={ H} \ltimes { N}$. For $x \in {  F}$, and  $a\in { F}^*$ we shall write
$$u(x)=\left( \begin{array}{cc} {1} & {x} \\ {0} & {1} \end{array} \right), \quad  h(a)=\left( \begin{array}{cc} {a} & {0} \\ {0} & {a^{-1}} \end{array} \right).$$
The map $u(x)\mapsto \psi(x)$ defines a character of $N$ we shall continue to denote it by $\psi$.

Let $\mslt^{(n)}=\widetilde{{ SL}_2({ F})^{n}}$ be the topological central extension of $\slt$ by $\mu_n$ constructed by Kubota in \cite{Kub}. We have the short exact sequence
$$1\rightarrow \mu_n \rightarrow \msl  \rightarrow  G \rightarrow  1.$$

%When connivent we shall think of $\msl$ as the set $\slt \times \mu$ equipped with the multiplication $(g_1,\epsilon_1)(g_2,\epsilon_2)=(g_1g_2,\epsilon_1\epsilon_2c(g_1,g_2)$ where $c(\cdot,\cdot):\slt \times \slt \rightarrow \mu_n$ is the usual cocycle, see for example \cite{}.
We shall denote by $\widetilde{A}$ the inverse image in $\widetilde{G}$ of a subgroup $A$ of $G$. A (complex) representation of $\mslt$ or any of its subgroups containing $\mu_n$ is called genuine if the central subgroup $\mu_n$ acts by the previously fixed injective character $\mu_n \hookrightarrow \C^*.$

%$\widetilde{G}$ splits uniquely over $N$. We shall identify $N$ with its image in  $\widetilde{G}$. For a representation $\pi$ of $\msl$. We denote by $Wh_\psi(\pi)$ the space of  $\psi$-Whittaker functional on a $\pi$. Namely, $Wh_\psi(\pi)=Hom_N(\pi,\psi)$. Unlike the linear case, the space of Whittaker functionals on smooth admissible irreducible genuine representations of $\widetilde{G}$ may be greater than 1. If $\tau$ is another representation of $\msl$ then any $A \in Hom_\msl{\pi,\tau)$ gives rise to $A^\psi$

%Give all nessery info includng action of $GL_2$. Use $\widetilde{G}$ and $\msl$. say about derived Say about cocycle.
\subsection{The metaplectic torus} \label{torussec}

If $n>2$, $\widetilde{H}$ is not abelian. Rather it is a 2-step nilpotent group whose center has a finite index. We note that $Z(\widetilde{H})$, the center of $\widetilde{H}$,  is $\widetilde{C}$ where
$$C=\{h(a) \mid a \in {F^*}^n \}.$$
%Since $n$ is odd $Z(\widetilde{H})$ is canonically isomorphic to $C\times \mu_n$.
By  the Stone-Von Neumann Theorem, see  Theorem 3.1 in \cite{Weissman09} for example, the isomorphism class of a genuine smooth irreducible  representation $\sigma$ of $\widetilde{{H}}$  is determined by its central character $\chi_\sigma$. Moreover, a realization of $\sigma$ is given by inducing from  $\chi'_\sigma$, a character of a maximal abelian subgroup of $\widetilde{H}$  which extends $\chi_\sigma$. The set of maximal abelian subgroups of $\widetilde{H}$ is in bijection with the set of maximal isotropic subgroups of $F^*/{F^*}^n$. Let $J$ be the pull-back of $\overline{J}$ to $F^*$. Set
$$M=\{h(a) \mid a \in J \}.$$
Then $\widetilde{M}$ is a maximal abelian subgroup of $\widetilde{H}$ and
$$[\widetilde{H}:\widetilde{M}]=[J:{F^*}^n]=\sqrt{[F^*:{F^*}^{n}]}=n \ab n \ab^{-\half}$$
See Section 3 in \cite{Sz19}. This implies that $\dim(\sigma)=n \ab n \ab^{-\half}$.

Since $n$ is odd  $Z(\widetilde{H})$ is canonically isomorphic to $C\times \mu_n$ and  $\widetilde{M}$  is canonically isomorphic to $M\times \mu_n$. This implies that set of genuine characters of  $Z(\widetilde{H})$ is canonically parameterized by the set of characters of ${F^*}^n$ and that the set of characters of $\widetilde{M}$ is canonically parameterized by the set of characters of $J$, see Section 4.2 in \cite{Sz19}.

The set of characters of ${F^*}^n$ is naturally identified with the set of $n^{th}$ powers of characters of $F^*$. Thus, given a genuine smooth irreducible  representation $\sigma$ of $\widetilde{{H}}$ we may associate to it
a unique character $\chi^n$ where $\chi$ is a character of $F^*$. We may think of $\chi^n$ as the restriction of $\chi$ to ${F^*}^n$. Note that $\sigma$ is unitary if and only if $\chi$ is unitary. We shall identify the restriction of  $\chi$ to $J$ with a genuine character of $\mm$ and we continue to denote this genuine character by $\chi$.

For $t\in \widetilde{H}$, an inverse image of $h(a)$ we set $\ab t \ab=\ab a \ab$. We now define $\sigma_s$ to be the representation of $\widetilde{H}$ given by $t\mapsto \ab t \ab^s \sigma(t)$. If $\chi^n$ is the character of $F^*$ associated with $\sigma$ then $(\chi_s)^n$ is the character of $F^*$ associated with $\sigma_s$.

Fix $w$, an inverse image of $w_{_0}=\left( \begin{array}{cc} {0} & {1} \\ {-1} & {0} \end{array} \right)$ in $\msl$. Since $w$ normalizes $\widetilde{H}$ and $\mu_n$ is central in $\msl$,  the automorphism of $\widetilde{H}$ defined by $t\mapsto wtw^{-1}$ is independent of the choice of an inverse image of $w_o$. We define  $\sigma^w$ to be a representation of  $\widetilde{{H}}$ by setting
$$t \mapsto \sigma^w(t)=\sigma(wtw^{-1}).$$
We note that ${(\sigma^w)}^w=\sigma$, ${(\sigma_s)}^w={(\sigma^w)}_{-s}$ and that if $\chi^n$ is associated  $\sigma$ then $\chi^{-n}$ is associated with $\sigma^w$.

\subsection{Genuine Principal series representations} \label{prinsec}
$\widetilde{G}$ splits uniquely over $N$. We shall identify $N$ with its image in  $\widetilde{G}$. Moreover, $\widetilde {B}=\widetilde{H} \ltimes {N}$. Therefore, as in the linear case, any representation $(\sigma,V_\sigma)$ of $\widetilde{H}$ can be extended to a representation of $\widetilde{{ B}}$ by defining it to be trivial on  ${ N}$. We shall not distinguish between representations of $\widetilde{{H}}$  and those of $\widetilde{{ B}}$. We consider the normalized parabolic induction
$$I(\sigma,s)=Ind^{\mslt}_{\mb} \sigma_{s+1}$$
and we set  $I(\sigma)=I(\sigma,0)$.

Fix $f \in { I}(\sigma)$ and $s \in \C$. Given $g \in \widetilde{{G}}$ we pick $t \in \widetilde{{H}}, \, n \in { N}$ and $k$ an inverse image inside $\widetilde{{G}}$ of an element of ${ SL}_2(\Of)$ such that $g=tnk$ and we set
\begin{equation} \label{flatsec} f_s(g)=\ab t \ab^sf(g). \end{equation} As in the linear case one verifies, using the Iwasawa decomposition of $G$, that $f_s$ is a well defined element in ${ I}(\sigma,s)$ and that $f\mapsto  f_s$ is an isomorphism of vector spaces. Similar to the linear case we have the following:
\begin{lem} \label{inter} For $f\in { I}(\sigma)$ and $g \in \widetilde{{G}}$ the integral
\begin{equation} \nonumber \int_F f_s\bigl(w u(x)g \bigr) \, d_{\psi}x \end{equation} converges absolutely to a $V_\sigma$ valued rational function in $q^{-s}$ provided that $\ab \chi^n(\varpi)\ab <q^{Re(s)n}$. Here $\chi^n$ is the character of $F^*$ associated with $\sigma$. We shall denote its meromorphic continuation by  $\bigl(A(\sigma,s,\psi)(f_s)\bigr)(g)$. Away from its poles,

$$A(\sigma,s,\psi) \in \operatorname{Hom}_{\mslt} \bigl({ I}(\sigma,s) ,{ I}(\sigma^w,-s) \bigr).$$
\end{lem}
For the proof see  Section 7 of \cite {Mc} for example.

We note that $A(\sigma,s,\psi)$ depends on $\psi$ only via the choice of the Haar-measure and that since $d_{\psi_c}=\ab c \ab^{\half} d_\psi$ we have
\begin{equation} \label{interpsi} A(\sigma,s,\psi_c)=\ab c \ab^{\half}A(\sigma,s,\psi). \end{equation}
While this dependence is typically suppressed in the literature we chose to indicate it in this paper.

\subsection{The space of Whittaker functionals} \label{whidefsec}
A $\psi$-Whittaker functional on a representation $(\pi,V_\pi)$ of $\msl$ is a functional on $V_\pi$ satisfying
$$\xi\bigl(\pi \bigl(u \bigr)v \bigr)=\psi(u)\xi(v)$$
for all $v \in V_\pi$, $u\in N$. Denote by $Wh_\psi(\pi)=Hom_N(\pi,\psi)$ the space of $\psi$-Whittaker functionals on $\pi$. If $\tau$ is another representation of $\msl$ then any
$A \in Hom_{\msl}(\pi,\tau)$ gives rise to a linear map $$A^\psi: Wh_\psi(\tau) \rightarrow Wh_\psi(\pi)$$  defined by $\xi \mapsto \xi \circ A$. We note that if $\pi=\tau$ and $A$ is the identity map then $A^\psi$ is also the identity map.

As noted in the introduction, unlike the linear case, the space of Whittaker functionals on smooth admissible irreducible genuine representations of $\widetilde{G}$ may be greater than 1. By Rodier Heredity, regardless of the question of reducibility, we have  $$\dim Wh_\psi\bigl(I(\sigma) \bigr)=n \ab n \ab^{-\half},$$ see Lemma 3.5 in \cite{Sz19} for example.

%Following the convention in the literature, we shall often write $A(\sigma,s)$ rather than $A(\sigma,s,\psi)$. MAYBE NOT.
\section{Main result} \label{mainresec}
\subsection{Plancherel measure} \label{planchsec}
Let $\sigma$ be  a smooth  genuine irreducible representation of $\mh$. There exists a rational function in $q^{-ns}$, $\mu_\psi(\sigma^w,s)$, such that
\begin{equation} \label{plangen} A(\sigma^w,-s,\psi) \circ A(\sigma,s,\psi)=\chi(-1)\mu_\psi^{-1}(\sigma,s)Id. \end{equation}
(here $Id$ is the identity map). See \cite{Sz19}, Section 5.1  for exact details. By Theorem 5.1 of \cite{GoSz} if $gcd(p,n)=1$ we have

\begin{equation} \label{planformula}\mu_\psi(\sigma,s)=\chi^n(-1)\gamma(1+ns,\chi^n,\psi_n)\gamma(1-ns,\chi^{-n},\psi_n) \end{equation}
Here $\chi^n$ is the character of $F^*$ associated with $\sigma$ (one needs to use \eqref{epsilon twist and inv}). By Theorem 5.7 of \cite{Sz19}, if $\chi^n$ is unramified then \eqref{planformula} holds also for the case where $gcd(p,n)>1$ (note that $q^{e(\psi_n)}=\ab n \ab q^{e(\psi)}$).

The Knapp-Stein dimension theorem for covering groups proven in \cite{Cai} states that given that $\sigma$ is unitary, $I(\sigma)$ is reducible if and only if $\sigma \simeq \sigma^w$ and $\mu_n^{-1}(\sigma,s)$ is analytic at $s=0$. By studying  the  analytical properties of the Plancherel measure it was proven in Proposition 5.5 if \cite{Sz19} that if $\sigma$ is unitary then $I(\sigma)$ is reducible of and only if $\chi^n$ is a non-trivial quadratic character of $F^*/{F^*}^n$. Since $n$ is odd, any non-trivial quadratic character of $F^*/{F^*}^n$ may be extended to a unique non-trivial quadratic character of $F^*$.

For $\theta$, a non-trivial quadratic character of $F^*$, let  $\sigma=\sigma_{_\theta}$ be (the isomorphism class of) the genuine unitary representation of $\mh$ whose associated character is $\theta^n=\theta$. By the above, as in the linear case, the map $\theta \mapsto I(\sigma_{_\theta})$ is a bijection from the set of non-trivial quadratic characters of $F^*$ to the set of reducible genuine principal series representations of $\msl$ induced from a unitary data.

\begin{prop} \label{planprop}  Equation \eqref {planformula} holds for the case where $\chi^n=\chi=\theta$ is a non-trivial quadratic character of $F^*$.
%$$A((\sigma_{_\theta})^w,0) \circ A(\sigma_{_\theta},0)=\ab n \ab \gamma(1,\theta,\psi)^2 Id.$$
\end{prop}
\begin{proof} We only need to consider the case where $gcd(n,p)>1$ and $\theta$ is ramified. Note that in this case $p\neq 2$ as $n$ is odd. Consequently, $e(\theta)=1$ (see Lemma 4.15 in \cite{Sz19} for example). Since $gcd(2,n)=1$,  for  any ramified element in $\eta \in \widehat{F^*/ {F^*}^n}$ we have  $e(\theta\eta)=e(\eta)$.  By Theorem 5.7 in \cite{Sz19} we only need to prove that
$$[F^*:{F^*}^n]^{-1} \sum_{\eta \in \widehat{F^*/ {F^*}^n}}q^{-e(\theta\eta)}=\ab n \ab q^{-1}.$$
Indeed, the $n$ unramified elements in $ \widehat{F^*/ {F^*}^n}$ contribute $nq^{-1}$ the sum above.
Thus,
$$[F^*:{F^*}^n]^{-1} \sum_{\eta \in \widehat{F^*/ {F^*}^n}}q^{-e(\theta\eta)}=n^{-2} \ab n \ab \Bigl(nq^{-1}+
\sum_{\eta \in \widehat{F^*/ {F^*}^n},  \eta \operatorname{is} \,  \operatorname{ramified. }} q^{-e(\eta)} \bigr).$$
From the identify in last line of Page 156 of \cite{Sz19} it follows that  $$\sum_{\eta \in \widehat{F^*/ {F^*}^n},  \eta \operatorname{is} \,  \operatorname{ramified. }} q^{-e(\eta)}=n(n-1)q^{-1}.$$ The proof is now complete.
\end{proof}
\subsection{Induction by stages}
Recall that we may realize $\sigma$ as $Ind_{\mm}^{\mh} \chi_\sigma'$ where $\chi_\sigma'$ is a character of $\mm$ whose restriction to $Z(\widetilde{H})$ equals $\chi_\sigma$. Recall also that  if $\chi^n$ is the character of $F^*$ associated with $\sigma$ then we may identify  $\chi_\sigma'$ with the restriction of $\chi$ to $J$. Using induction by stages we may realize $I(\sigma,s)$ as
$$I(\chi,s)=Ind^\mslt_{N \mm}  \chi_{s+1}.$$
Note that by Stone-Von Neumann theorem, if $\eta$ is a character of $F^*$ whose restriction to ${F^*}^n$ is trivial (equivalently, $\eta \in \widehat{{F^*}/{F^*}^n}$) then $I(\chi,s) \simeq (\chi\eta,s)$. When $s=0$ we write  $I(\chi)$ instead of $I(\chi,0)$.

Using this realization, away from its poles, the intertwining operator $A(\chi,s,\psi)$ takes the form
$$A(\chi,s,\psi) \in \operatorname{Hom}_{\mslt} \bigl({ I}(\chi,s) ,{ I}((\chi^{-1},-s) \bigr).$$
See the commutative diagram in Page 137 of \cite{Sz19} for a proof. That diagram also implies that
%As before we shall write $A(\chi,s)$ instead of $A(\chi,s,\psi)$ whenever connivent. We have
$$A(\chi^{-1},-s,\psi) \circ A(\chi,s,\psi)=\chi(-1)\mu_\psi^{-1}(\sigma,s)Id .$$

In the special case where $\chi= \theta$ is a non-trivial quadratic character of $F^*$,   $A(\theta,0,\psi)$ is a self intertwining operator. Moreover, by the theory of $R$-groups, it is not a scalar operator, see \cite {Cai}. Using Proposition \ref{planprop} and Equation \eqref{changepsi} we obtain
\begin{equation} \label{planeq}A(\theta,0,\psi)^2=\ab n \ab \gamma(1,\theta,\psi)^2Id. \end{equation}

\subsection{A scattering matrix} \label{scatsec}
In Section 4.2 of \cite{Sz19} we have introduced a  basis  $\{\lambda_{b,\chi,\psi,s} \mid b\in \overline{K}\}$ for $Wh_{\psi}\bigl(I(\chi,s)\bigr)$ and a function
$$\tau(\cdot,\cdot,\chi,s,\psi):\overline{K} \times \overline{K} \rightarrow \C[q^{-s}].$$
defined by the relation
\begin{equation} \label{tauprop} A^\psi(\chi,s,\psi)\bigl(\lambda_{a^{-1},\chi^{-1},\psi,-s}\bigr)=\sum_{b \in \overline{K}} \tau(a,b,\chi_{},s,\psi)\lambda_{b,\chi,\psi,s}.\end{equation}
See \cite{Sz19}, Proposition 4.9. The appearance of $a^{-1}$  rather than $a$ on the left hand side is a manifestation of the fact that both Whittaker spaces $Wh_{\psi}\bigl(I(\chi,s)\bigr)$ and   $Wh_{\psi}\bigl(I(\chi^{-1},-s)\bigr)$ were identified in a canonical way with  the space of functionals on $\sigma$. This identification gave rise to the local coefficients matrix whose characteristic polynomial is an invariant of $\sigma$ and $\psi$, see Section 3.2 of \cite{GSS2}.

However, if $\chi=\theta$ is a non-trivial quadratic character of $F^*$ then $A^{\psi_c}(\theta,0,\psi)$ is an endomorphism on the space $Wh_{\psi_c}\bigl(I(\theta,0)\bigr)$ and we wish to change our point of view. In this special case it follows from the above that if we define
$$M(\cdot,\cdot,\chi,\psi):\overline{K} \times \overline{K} \rightarrow \C[q^{-s}]$$
by $$M(a,b,\theta,\psi)=\tau(a^{-1},b,\theta,0,\psi).$$
then
 \begin{equation} \label{whym}A^\psi(\theta,0,\psi)\bigl(\lambda_{a,\theta,\psi,0}\bigr)=\sum_{b \in \overline{K}} M(a,b,\theta,\psi)\lambda_{b,\theta,\psi,0}. \end{equation}
In other words, $M(\cdot,\cdot,\theta,\psi)$ represents the endomorphism $A^\psi(\theta,0,\psi)$.
\begin{remark} Suppose that $gcd(n,p)=1$ and  that $\theta$ is the unique non-trivial quadratic unramified character of $F^*$. If one fixes $\overline{J}$ as in Remark \ref{stanrem} then the conjugacy class of $M(\cdot,\cdot,\chi,\psi)$ is the same as the conjugacy class of the scattering matrix computed in \cite{Mc2} as both matrices represents the same map.
\end{remark}
\begin{lem} \label{compm} Let $\theta$ be a non-trivial quadratic character of $F^*$. We have
$$M(a,b,\theta,\psi)=\gamma_{_J}(1,\theta\eta_{a^{-1}b},\psi,(ab)^{-1})$$
 \end{lem}
\begin{proof} This follows at once from the coarse description of $\tau(a,b,\chi,0,\psi)$ given in \cite{Sz19}, Theorem 4.12.
 \end{proof}
\begin{prop} \label{trprop} Assume that $\theta$ is a non-trivial quadratic character of $F^*$. We have
$$\operatorname{tr }A^{\psi_c}(\theta,0,\psi)=\theta(c)\gamma(1,\theta,\psi).$$
%\begin{enumerate}
%\item $\operatorname{tr }A^{\psi}(\theta,0,\psi)=\gamma(1,\theta,\psi)$.
%\item $\operatorname{tr }A^{\psi_c}(\theta,0,\psi)=\theta(c)\gamma(1,\theta,\psi)$.
%\end{enumerate}
\end{prop}
\begin{proof}
Assume first that $c=1$. From \eqref{whym} it follows that  $$\operatorname{Tr }A^{\psi}(\theta,0,\psi)=\sum_{a \in \overline{K}}M(a,a,\theta,\psi).$$
Lemma \ref{compm} gives
$$\sum_{a \in \overline{K}}M(a,a,\theta,\psi)=\sum_{a \in \overline{K}}\gamma_{_J}(1,\theta,\psi,a^{-2})=
\sum_{k \in \overline{K}}\gamma_{_J}(1,\theta,\psi,k).$$
For the last equality we have used the fact that since $\overline{K}$ is an abelian group of odd order, $a\mapsto a^{-2}$ is an automorphism of $\overline{K}$. With  \eqref{invert} the proof of for the case where $c=1$ is completed.
%$$\sum_{k \in \overline{K}}\gamma_{_J}(1,\chi,\psi,k)={(\# \overline{J})}^{-1} \sum_{k \in \overline{K}}  \sum_{j \in \overline{J}} \gamma(1,\chi\eta_j,\psi)\eta_j(k^{-1})=
%{(\# \overline{J})}^{-1}   \sum_{j \in \overline{J}} \gamma(1,\chi\eta_j,\psi)\Bigl(\sum_{k \in \overline{K}}\eta_j(k^{-1})\Bigr)=\gamma(1,\chi,\psi).$$

 We now move to the general case. From \eqref{interpsi} it follows that
$\ab c \ab^{-\half}M(\cdot,\cdot,\theta,\psi_c)$ represents  $A^{\psi_c}(\theta,0,\psi)$. Thus,
  $$\operatorname{Tr }A^{\psi_c}(\theta,0,a)=\ab c \ab^{-\half} \sum_{a \in \overline{K}}M(a,a,\theta,\psi_c)=\ab c \ab^{-\half} \gamma(1,\theta,\psi_c).$$
With \eqref{changepsi} we finish.
\end{proof}
\subsection{Whittaker dimensions } \label{mainsec}
From this point we shall write $\pi=I(\sigma_{_\theta})$. We define a normalized intertwining operator
\begin{equation} \label{noreq} \mathbb{A}(\theta,\psi)=\ab n \ab^{-\half}\gamma(1,\theta,\psi)^{-1} A(\theta,0,\psi). \end{equation}
From \eqref{planeq} it follows that $\mathbb{A}(\theta,\psi)^2=Id$. Since $\mathbb{A}(\theta,\psi)$ is not a scalar operator we may write  $I(\theta)$ as a sum of two irreducible representations of $\msl$
$$I(\theta)=  \pi^+_\psi \oplus \pi^-_\psi$$
where $\pi^{\pm}_\psi$ is the eigenspace  of $\mathbb{A}(\theta,\psi)$ corresponding to $\pm 1$. We note that since $\mathbb{A}(\theta,\psi)^2$ is the identity map it follows that $\half(I \pm \mathbb{A}(\theta,\psi))$ are the projections from $I(\theta)$ to  $\pi^{\pm}_\psi$.
\begin{thm} \label{mainthm}  Let $\theta$ be a non-trivial quadratic character of  $F^*$. Denote $\pi=I(\sigma_{_\theta})$. We have
$$\dim Wh_{\psi_c} (\pi^{+}_\psi)=\ab n \ab^{-\half} \frac{n+\theta(c)}{2}, \, \, \dim Wh_{\psi_c} (\pi^{-}_\psi)=\ab n \ab^{-\half} \frac{n-\theta(c)}{2}.$$
In particular,
 $$\dim Wh_\psi (\pi^{+}_\psi)=\ab n \ab^{-\half} \frac{n+1}{2}, \, \, \dim Wh_\psi (\pi^{-}_\psi)=\ab n \ab^{-\half} \frac{n-1}{2}.$$

\end{thm}
\begin{proof} Since $\half(Id \pm \mathbb{A}(\theta,\psi))^{\psi_c}$ are the projections from $ Wh_{\psi_c}\bigl(I(\theta)\bigr)$ to  $ Wh_{\psi_c} (\pi^{\pm}_\psi)$ we obtain
\begin{eqnarray}  \nonumber \dim Wh_{\psi_c} (\pi^{\pm}_\psi) &=& \operatorname{rank } \half(Id \pm \mathbb{A}(\theta,\psi))^{\psi_c}=\operatorname{tr } \half(Id \pm \mathbb{A}(\theta,\psi))^{\psi_c}\\ \nonumber  &=& \half \bigl(n \ab n \ab^{-\half}\pm\ab n \ab^{-\half}\gamma(1,\theta,\psi)^{-1}\operatorname{tr }A^{\psi_c}(\theta,0,\psi)\bigr)
\end{eqnarray}
% $$\dim Wh_{\psi_c} (\pi^{\pm}_\psi)=\operatorname{rank } \half(I \pm \mathbb{A}(\theta,\psi))^{\psi_c}=\operatorname{tr } \half(I \pm \mathbb{A}(\theta,\psi))^{\psi_c}
% =\half \bigl(n \ab n \ab^{-\half}\pm\ab n \ab^{-\half}\gamma(1,\theta,\psi)^{-1}\operatorname{tr }A^{\psi_c}(\theta,0,\psi)\bigr)$$
With Proposition \ref{trprop} we finish.

\end{proof}
\begin{remark} \label{mainrem} Since $\dim Wh_\psi(\pi)$ is odd we know a-priori that the $\psi$-Whittaker dimensions of the two irreducible summands inside  $\pi$ are not equal. From the last Theorem it follows that we can define $\pi_\psi^\pm$ by stating that \begin{equation} \label{generic} \dim Wh_\psi (\pi^{+}_\psi)>\dim Wh_\psi (\pi^{-}_\psi).\end{equation} This second characterization of $\pi_\psi^\pm$ shows that the definition of these sub-representations as eigenspaces is free of the choice of $\overline{J}$ giving rise to the realization $I(\theta)$. The action of the intertwining operator on the two summands inside $\pi$ is analogous to   the formula for the action of the intertwining operator in the case of quasi-split reductive algebraic groups given in \cite{KS}. That formula involved the $R$-groups and had the generic summand as its base point. Here the $R$-group is the group of two elements and due to \eqref{generic}, $\pi^{+}_\psi$ plays the role of the generic summand.
\end{remark}

\begin{remark} \label{Sharem} Let $\sigma$ be an irreducible genuine representation of $\mh$ and let $\chi^n$ the character of $F^*$ associated with $\sigma$. In Section 3.5.1 of \cite{GSS} we have defined an invariant $S_w(\sigma,s,\psi)$ using the determinant  of the local coefficient matrix and the Plancherel measure and made the point that it is an analog of the reciprocal of Shahidi local coefficients defined for quasi-split reductive algebraic groups, \cite{Shabook}. In particular, it follows at once from the definition that
\begin{equation} \label{shainvplan} S_w(\sigma,s,\psi)S_w(\sigma^w,-s,\psi)=\chi^n(-1)\mu_\psi^{-1}(\sigma,s).\end{equation}
Note that Equation \eqref{shainvplan} is unchanged if we replace $S_w(\sigma,s,\psi)$ by $$S_w'(\sigma,s,\psi)=\chi^n(n)S_w(\sigma,s,\psi).$$ In the case where $gcd(n,p)=1$ and $\sigma$ is unramified it was shown in Theorem 3.14 of \cite{GSS} that $S_w(\sigma,s,\psi)=\gamma(1-ns,\chi^{-n},\psi_n).$ This last result was recently extended to all $\sigma$ under the additional assumption that $p$ is odd, see the introduction of  \cite{Sz24}. Thus, if $p$ is odd and relatively prime to $n$ then from \eqref{changepsi} it follows that the normalization factor $\ab n \ab^{\half}\gamma(1,\theta,\psi)$ we have used in \eqref{noreq}  for $A(\theta,0,\psi)$ equals $S_w'(\sigma_\theta,0,\psi)$. It is expected of course that these factors are equal in the other cases as well: note that from \eqref{shainvplan}, Proposition \ref{planprop} and Equation \eqref{changepsi} it follows that $S_w'(\sigma_\theta,0,\psi)^2=\ab n \ab \gamma(1,\theta,\psi)^2$. Thus, the factors are equal at least up to a sign for all $p$.
\end{remark}

\section{The action of $\mgl$} \label{glsec}
Let $\mgl$ be the $c$-twisted central extension of $\glt$ by $\mu_n$ introduced in \cite{KP}. The derived group of $\mgl$ is $\msl$. Therefore,  $\mgl$ is acting on $\msl$ by conjugations. For $h\in \msl$ and $g\in \widetilde{G}$ set $h^g=g^{-1}hg$.

Since $\mu_n$ is central in $\mgl$ this action factors through $\glt$ so for $h\in \msl$ and $g\in \glt$ we can write $h^g$ with no ambiguity. We also have an action of $\widetilde{G}$ on the set of isomorphism classes of representations of $\msl$. For a representation $\tau$ of  $\msl$ and $g\in \glt$ we define a representation $\tau^g$ of  $\msl$ by $h\mapsto \tau(h^g)$.

The centralizer of $\msl$ in $\mgl$ is $\widetilde{S}$, the inverse image in $\mgl$ of the group of scalar matrices in $\glt$, see Lemma 8.2 in \cite{GSS2}. Thus, the action of $\mgl$ on the set of isomorphism classes of representations of $\msl$ reduces to an action of  $\mgl/\widetilde{G^+}$ where $\widetilde{G^+}$ is the inverse image in $\mgl$ of
$$G^+=SG=\{g\in G\mid det(g)\in {F^*}^2 \}.$$
Let  $g \in \mgl$ be an inverse image in $\mgl$ of $g' \in \glt$. Set $\det(g)=\det(g')$. The map $g \mapsto \det(g)$ gives rise to an isomorphism from $\widetilde{G}/ \widetilde{G^+}$ to $F^*/{F^*}^2$. As a set of representatives of  $\widetilde{G}/ \widetilde{G^+}$ we can take a  set of inverse images of
$$\{diag(1,c) \mid c\in F^*/{F^*}^2\}.$$
We note that the group of inverse images in $\mgl$ of the diagonal subgroup in $\glt$ normalizes $N$. In particular for $u\in N$ we have \begin{equation} \label{conjpsi} \psi\bigl(u^{diag(1,c)}\bigr)=\psi_c(u). \end{equation}

\begin{lem} \label{fix} Let $\tau$ be a genuine principal series representation of $\msl$. For any $g \in  \mgl$, $\tau^g \simeq \tau.$
\end{lem}
\begin{proof}  We realize $\tau$ as $I(\chi)$, an induction from $\widetilde{M}N$. Let $g\in \mgl $ be an inverse image of $g'=diag(1,c)$ and let $m \in \mm$ be an inverse image of $m'\in M$. Since $g'$ commutes with $m'$ there exists a character $\eta:M \rightarrow \mu_n$ such that  $m^g= \eta(m')m$. The map  $T:I(\chi)^g \rightarrow I(\chi\eta)$ defined by $T(f)(h)=f(h^g)$ is a $\msl$ isomorphism. Finally, by Lemma 8.3 in \cite{GSS2}, since $n$ is odd, $Z(\widetilde{H})= Z(\widetilde{T}) \cap \msl$. Consequently, the restriction of $\eta$ to $Z(\widetilde{H})$ is trivial. From the  Stone-Von Neumann theorem it follows that $I(\chi\eta) \simeq   I(\chi)$.
\end{proof}

\begin{remark} Lemma \ref{fix} is false for even fold covers of $G$ since in this case $Z(\widetilde{T}) \cap \msl \subsetneq Z(\widetilde{H})$.
\end{remark}
\begin{thm} \label{swapthm} Let $\theta$ be a non-trivial quadratic character of $F^*$. Denote $\pi=I(\sigma_{_\theta})$.  Fix  $g \in \mgl$.
\begin{enumerate}
\item If $\theta(\det g)=1$ then $\bigl( \pi^+_\psi \bigr)^g \simeq  \pi^+_\psi$ and $\bigl( \pi^-_\psi \bigr)^g \simeq  \pi^-_\psi$.
\item If $\theta(\det g)=-1$ then $\bigl( \pi^+_\psi \bigr)^g \simeq  \pi^-_\psi$ and $\bigl( \pi^-_\psi \bigr)^g \simeq  \pi^+_\psi$.
\end{enumerate}
\end{thm}

\begin{proof} From Lemma \ref{fix} it follows that either $g$ fixes each of the 2 irreducible sub representations of $\pi$ or it swaps them. Thus, we only need to determine the isomorphism class of ${(\pi^+_\psi)}^g$.

As before we may assume that $g$ is an inverse image of $diag(1,c)$. In light of Theorem \ref{mainthm} and Remark \ref{mainrem} it is sufficient to show that $Wh_\psi(\pi^+)=Wh_{\psi_c}\bigl({(\pi^+_\psi)}^g \bigr)$. Indeed if $\xi \in Wh_\psi(\pi^+)$ then from \eqref{conjpsi} we deduce that for $u \in N $ we have
$$\xi \circ \pi^g(u)=\psi_c(u)\xi.$$ This shows that $Wh_\psi(\pi^+) \subseteq Wh_{\psi_c}\bigl({(\pi^+_\psi)}^g \bigr)$. The other inclusion is proven in the same way.
\end{proof}
\section{The unramified case} \label{unramsec}
In this section we shall use Theorem \ref{mainthm} and Theorem \ref{swapthm} to reprove and slightly extend the results in Section 5.4.1 of \cite{GSS3}.

Assume that $gcd(n,p)=1$. In this case $\msl$ splits over $SL_2(\Of)$ uniquely. Denote by $K_1$ the image of this group in $\msl$. Denote also $K_2=K_1^{diag(1,\varpi)}$.

Let $\sigma$ be an irreducible genuine representation of $\widetilde{H}$. Suppose that $\sigma$ is an irreducible representation associated with a character $\chi^n$ of $F^*$. $I(\sigma)$ is called unramified with respect to $K_i$ if it has a non-zero fixed $K_i$ vector. As in the linear case, due to the  Iwasawa decomposition, if $I(\sigma)$ is unramified then it has a one dimensional $K_i$ fixed subspace. We note that $I(\sigma)$ is unramified with respect to $K_1$ if and only if $\chi^n$ is unramified, see Section 8.6 in \cite{GSS2}. Since  $I(\sigma)$ is ramified with respect to $K_1$ if and only if $I(\sigma)^{diag(1,\varpi)}$ is ramified with respect to $K_2$ it follows from Lemma \ref{fix} that $I(\sigma)$ is unramified with respect to $K_1$ if and only if it is unramified with respect to $K_2$. We remark here that for even fold covers of $G$, a genuine principal series representation can not be unramified with respect to both $K_1$ and $K_2$.

Note that $F^*$ has exactly one non-trivial unramified quadratic charter, namely $x\mapsto \theta_u(x)=(-1)^{val(x)}$. Thus, there exists a unique genuine reducible unramified principal series representation $\pi$ of $\msl$ induced from a unitary representation. For $i=1,2$ let $v_i$ be the unique (up to normalization) $K_i$ spherical  vector and let $V_i$ be the $\msl$ invariant subspace of $\pi$ generated by $v_i$.

If  $e(\psi)$=0 then from  Proposition 7.3 in \cite{Gao18} we deduce that $\mathbb{A}(\theta,\psi)v_1=v_1$. Hence if  $e(\psi)$=0 then  $V_1={\pi}_\psi^+$.

\begin{thm} Assume that $gcd(n,p)=1$. Let ${\pi}$ be the unique genuine unitary reducible unramified principal series representation  of $\msl$. For $i=1,2$ let $v_i$ be the unique (up to normalization) $K_i$ unramified vector and let $V_i$ be the $\msl$ sub-representation of ${\pi}$ generated by $v_i$. We have
\begin{enumerate}
\item If $e(\psi)$ is even then $V_1={\pi}_\psi^+$ and $V_2={\pi}_\psi^-$.
\item If $e(\psi)$ is odd then $V_1={\pi}_\psi^-$ and $V_2={\pi}_\psi^+$.
\end{enumerate}
\end{thm}
\begin{proof}
Since we already know that if $e(\psi)=0$ then $V_1={\pi}_\psi^+$ and since $ e(\psi)\equiv e(\psi_c) \, (\operatorname{mod }2)$ if and only if $\theta_u(c)=1$ it follows from Theorem \ref{mainthm} and Remark \ref{mainrem} that if $e(\psi)$ is even then $V_1={\pi}_\psi^+$ and that if $e(\psi)$ is odd then $V_1={\pi}_\psi^-$. To finish the proof we note that $v_2 \in V_1^{diag(1,\varpi)}$ and that since $\theta_u(\varpi)=-1$ it now follows from Theorem \ref{swapthm} that if $e(\psi)$ is even then $V_2={\pi}_\psi^-$ and that if $e(\psi)$ is odd then $V_2={\pi}_\psi^+$.
\end{proof}
\begin{cor} $$\mathbb{A}(\theta_u,\psi)v_1=(-1)^{e(\psi)}v_1 , \quad \mathbb{A}(\theta_u,\psi)v_2=(-1)^{e(\psi)+1}v_2.$$
\end{cor}

{\bf Acknowledgment.} I would like to thank Fan Gao and Nadya Gurevich for numerous enlightening communications.

{\bf Funding.} This research was supported by the Israel Science Foundation (grant No. 1643/23).

\bibliography{dani}
\bibliographystyle{acm}
\end{document}